 \documentclass[final,1p,times]{elsarticle}
 \usepackage{enumitem}

\usepackage{amssymb}
\usepackage{amsthm}
\usepackage{amsmath}
\usepackage{MnSymbol}
\usepackage{lineno}
\newtheorem{theorem}{Theorem}[section]
\newtheorem{lemma}[theorem]{Lemma}

\theoremstyle{definition}
\newtheorem{definition}[theorem]{Definition}

\theoremstyle{proposition}
\newtheorem{proposition}[theorem]{Proposition}

\theoremstyle{remark}

\theoremstyle{notation}

\theoremstyle{conjecture}
\newtheorem{conjecture}[theorem]{Conjecture}

\theoremstyle{corollary}
\newtheorem{corollary}[theorem]{Corollary}

\numberwithin{equation}{section}

\journal{Symmetry}

\begin{document}

\begin{frontmatter}




\title{An E-sequence approach to  the 3x + 1 problem}

\author{SanMin Wang}
\address{Department of Mathematics,  Zhejiang Sci-Tech University, Hangzhou,
310018, P.R. China\\
Email: wangsanmin@hotmail.com}

\date{}
\begin{abstract}
For any odd positive integer $x$,  define $(x_n)_{n\geqslant 0} $ and $(a_n )_{n\geqslant 1} $  by
setting $x_{0}=x, \,\, x_n =\cfrac{3x_{n-1} +1}{2^{a_n }}$
such that all $x_n $  are odd.  The 3x+1 problem asserts that
there is  an $x_n =1$  for all $x$. Usually,  $(x_n )_{n\geqslant 0} $ is called the trajectory of $x$.  In this paper, we  concentrate on $(a_n )_{n\geqslant 1} $  and call it the E-sequence of $x$.   The idea is that,
we generalize E-sequences  to all infinite sequence $(a_n )_{n\geqslant 1} $  of positive integers and consider  all these generalized E-sequences.   We then define $(a_n )_{n\geqslant 1} $  to be  $\Omega-$convergent to $x$  if it is the E-sequence of  $x$ and to be $\Omega-$divergent if it is not the E-sequence of any odd positive integer.  We prove a remarkable fact  that the $\Omega-$divergence of all non-periodic E-sequences  implies the periodicity of  $(x_n )_{n\geqslant 0} $  for all $x_0$.   The principal results of this paper are to prove the $\Omega-$divergence of several classes of non-periodic E-sequences.  Especially, we prove that all non-periodic E-sequences $(a_n )_{n\geqslant
1} $ with $\mathop {\overline {\lim } }\limits_{n\to \infty } \cfrac{b_n
}{n}>\log _23$ are $\Omega-$divergent by using the Wendel's inequality and the Matthews and Watts's formula $x_n
=\cfrac{3^n x_0 }{2^{b_n }}\prod\limits_{k=0}^{n-1} {(1+\cfrac{1}{3x_k })} $,
where $b_n =\sum\limits_{k=1}^n {a_k } $.
These results present  a possible way to prove
 the periodicity of  trajectories of all positive integers in the 3x + 1 problem
and  we call it the E-sequence approach.
\end{abstract}


\end{frontmatter}
\setlength\linenumbersep{1cm}
\renewcommand\linenumberfont{\normalfont}

\section{Introduction}
For any odd positive integer $x$, define two infinite sequences $(x_n
)_{n\geqslant 0} $ and $(a_n )_{n\geqslant 1} $ of positive integers by
setting
\begin{equation}x_0 =x, \quad x_n =\cfrac{3x_{n-1} +1}{2^{a_n }}
\end{equation}
such that $x_n $ is
odd for all $n\in \mathbb{N}=\{1,2,\ldots \}$. The 3x+1 problem asserts that
there is $n\in \mathbb{N}$ such that $x_n =1$  for all odd positive integer
$x$. For a survey,  see~\cite{La85}.  For recent developments,  see [9-14].

Usually,  $(x_n )_{n\geqslant 0} $ is called the trajectory of $x$.  In this paper, we  concentrate on $(a_n )_{n\geqslant 1} $  and call it the E-sequence of $x$.   The idea is that, we generalize E-sequences  to all infinite sequence $(a_n )_{n\geqslant 1} $  of positive integers.  Given any generalized E-sequence $(a_n )_{n\geqslant 1} $,  if it is the E-sequence of the odd positive integer $x$, it is called to be $\Omega-$convergent to $x$ and, denoted by $\Omega -\lim a_n =x$; if
$(a_n )_{n\geqslant 1} $  is not the E-sequence of any odd positive integer, it is called to be $\Omega-$divergent and denoted by $\Omega -\lim a_n =\infty $. Subsequently, these generalized E-sequences are also called  E-sequences for simplicity.

The 3x+1 problem in the form  (1.1)  should be owed to  Crandall, Sander et al.,  see~\cite{C77, S90}.  E-sequences are some variants of Everett's parity sequences~\cite{Ev77} and Terras's encoding representations~\cite{Te76}. Everett and Terras  focused on
finite E-sequences resulted from (1.1). What we
concern is the $\Omega-$convergence and $\Omega-$divergence of any infinite sequence of
positive integers, i.e., the generalized E-sequences.

A possible way to prove the 3x+1 problem were devised by M\"{o}ller as
follows, see~\cite{Mo78}.
\begin{conjecture}\label{C:1}
\begin{enumerate}[label=\upshape(\roman*), leftmargin=*, widest=iii]
\item  $(x_n )_{n\geqslant 0} $  is periodic for all odd positive integer $x_0$;\label{it:1}
\item  $(1,1,\cdots )$ is the unique pure periodic trajectory.\label{it:2}
\end{enumerate}
\end{conjecture}
Usually, we can convert one claim about trajectories into the one about
E-sequences. As for E-sequences, we have the following conjecture.

\begin{conjecture}\label{C:2}
Let $b_n \mbox{=}\sum\limits_{i=1}^n {a_i } $.  Then
\begin{enumerate}[label=\upshape(\roman*), leftmargin=*, widest=iii]
\item all non-periodic E-sequences are $\Omega-$divergent;\label{it:1}
\item every E-sequence $(a_n )_{n\geqslant 1} $ satisfying $3^n >2^{b_n }$ for all $n\in \mathbb{N}$ is $\Omega-$divergent.\label{it:2}
\end{enumerate}
\end{conjecture}
Note that Conjecture~\ref{C:2}\ref{it:1} does not hold for some generalizations of the 3x+1 problem studied by M\"{o}ller,  Matthews and Watts in~\cite{MW84, Mo78}; Conjecture ~\ref{C:2}\ref{it:2} implies that there is some
$n$ such that $2^{b_n }>3^n$ in the E-sequence $(a_n )_{n\geqslant 1} $ of every odd
positive integer $x$, which is a conjecture posed by Terras in~\cite{Te76} about his
$\tau -$stopping time.

A remarkable fact is that Conjecture~\ref{C:1}\ref{it:1} is a corollary of Conjecture~\ref{C:2}\ref{it:1}
by Theorem 3.6.  This means that the $\Omega-$divergence of all non-periodic E-sequences  implies the periodicity of  $(x_n )_{n\geqslant 1} $  for all positive integers $x$.
Then Conjecture~\ref{C:2}\ref{it:1} is of significance to the study of the 3x+1 problem. The principal results of this paper are to prove that several classes of non-periodic E-sequences are $\Omega-$divergent. In particular, we prove
that
\begin{enumerate}[label=\upshape(\roman*), leftmargin=*, widest=iii]
\item All non-periodic E-sequences $(a_n )_{n\geqslant
1} $with $\mathop {\overline {\lim } }\limits_{n\to \infty } \cfrac{b_n
}{n}>\log _23$ are $\Omega-$divergent.

\item  If $(a_n )_{n\geqslant 0} $  is  $12121112\cdots $, where $a_n
=2$ if $n\in \{2^1,2^2,2^3,\cdots \}$ and $a_n =1$
otherwise, then $\Omega -\lim a_n =\infty $;

\item Let $\theta \geqslant 1$ be an irrational number, define $a_n =[n\theta ]-[(n-1)\theta
]$, then $\Omega -\lim a_n =\infty $, where $[a]$ denotes the integral part of $a$ for any real
$a$.
\end{enumerate}

 Note that we prove the above claim (i) by using the Wendel's inequality and the Matthews and Watts's formula $x_n
=\cfrac{3^nx_0 }{2^{b_n }}\prod\limits_{k=0}^{n-1} {(1+\frac{1}{3x_k })} $.
 In addition,  it seems that our approach cannot help to prove the conjecture 1.1(ii) of the unique cycle.  For such a topic, see~\cite{S08}.

\section{Preliminaries}
Let $(a_n )_{n\geqslant 1} $ be an E-sequence.  In most cases,  there is no odd
positive integer $x$ such that $(a_n )_{n\geqslant 1} $  is  the E-sequence of $x$, i.e., $\Omega -\lim a_n =\infty $.  However, there always exists $x\in \mathbb{N}$ such that the first $n$ terms of  the E-sequence of $x$  is  $(a_1\ldots a_n)$.  Furthermore, for any  $1\leqslant u\leqslant v\leqslant n$,  there always exists $x\in \mathbb{N}$ such that the first $v-u+1$ terms of  the E-sequence of $x$  is  the designated block $(a_u\ldots a_v)$ of $(a_1\ldots a_n)$, which  is  illustrated as  $(a_1\ldots a_{u-1}) (a_u\ldots a_v)(a_{v+1}\ldots a_{n})$.

\begin{definition}
 Define $b_0 =0$, $b_n
=\sum\limits_{i=1}^n {a_i }, B_n  = \sum\limits_{i = 0}^{n - 1}  \,3^{n - 1 - i} 2^{b_i } $.
\end{definition}

Clearly, $B_1 =1$, $B_n =3B_{n-1} +2^{b_{n-1} }$, $2\nmid B_n $,
$3\nmid B_n $.

\begin{proposition}
Let $(x_n )_{n\geqslant 1} $  and $(a_n )_{n\geqslant 1} $be defined as in (1.1). Then \\
 $x_n =\cfrac{3^nx+B_n }{2^{b_n }}$.
\end{proposition}
\begin{proof}
 The proof is by a procedure similar to that of Theorem 1.1
in~\cite{Te76}  and omitted.
\end{proof}

\begin{proposition}
Given any positive integer $n$, there exist two integers $x_n $ and $x_{0}$ such
that $2^{b_n }x_n -3^nx_{0}=B_n $, $1\leqslant x_n <3^n$ and $1\leqslant
x_{0}<2^{b_n }$.
\end{proposition}

\begin{proof}
By $\gcd (2^{b_n },3^n)=1$, there exist two integers $x_n $
and $x_{0}$ such that $2^{b_n }x_n -3^n x_{0}=B_n $ and $1\leqslant x_n \leqslant
3^n$. Then $x_n <3^n$ by $3\nmid B_n$. By $B_n \geqslant 1$,  we have
$x_{0}=\cfrac{2^{b_n }x_n -B_n }{3^n}<\cfrac{2^{b_n }x_n }{3^n}<2^{b_n }$.  Thus $x_{0}<2^{b_n }$.

By $2^{b_n }x_n -3^nx_{0}=B_n $,  we have $2^{b_n }x_n \equiv B_n
(\bmod 3^n)$. Then\\
$2^{b_{n-1} }(2^{a_n }x_n -1)\equiv 3B_{n-1}
(\bmod 3^n)$ by $B_n =2^{b_{n-1} }+3B_{n-1} $.
Thus $3\vert 2^{a_n }x_n -1$. Define $x_{n-1} =\cfrac{2^{a_n }x_n -1}{3}$.
Then $x_{n-1} \in \mathbb{Z}$, $x_n =\cfrac{3x_{n-1} +1}{2^{a_n }}$ and\\ $2^{b_{n-1 }}x_{n-1}\equiv B_{n-1}
(\bmod 3^{n-1})$.  Sequentially define $x_{n-2},  {\ldots}, x_1$ such that \\ $x_{n-1} =\cfrac{3x_{n-2} +1}{2^{a_{n-1} }}$, {\ldots}, $x_1 =\cfrac{3x_{0}+1}{2^{a_1 }}$. Then $x_i \in \mathbb{Z}$ for all
$0\leqslant i\leqslant n$.

Suppose that $x_{0}<0$.  We then sequentially have $x_1 <0,...,x_n <0$, which contradicts with $x_n \geqslant 1$. Thus $x_{0}\geqslant 1$.
\end{proof}

Note that the validity of Proposition 2.3 is dependent on the structure of
$B_n $. We formulate the middle part of the above proof as the following proposition.

\begin{proposition}
Assume that  $x_n, x_{0}\in \mathbb{Z}$ and $2^{b_n }x_n -3^nx_{0}=B_n $.  Define $x_1 =\cfrac{3x_{0}+1}{2^{a_1 }}$,{\ldots}, $x_{n-1} =\cfrac{3x_{n-2} +1}{2^{a_{n-1} }}$.   Then $x_{n} =\cfrac{3x_{n-1} +1}{2^{a_{n} }}$ and $x_i \in \mathbb{Z}$  for all
$0\leqslant i\leqslant n$.
\end{proposition}

\begin{definition}
For any $1\leqslant u\leqslant v$, define $b_u^{u-1} =0$, $b_u^v
=\sum\limits_{i=u}^v {a_i } $, $B_u^{u-2} =0$, $B_u^{u-1} =1$, $B_u^v
=3^{v-u+1}+3^{v-u}2^{b_u^u }+\cdots +3^12^{b_u^{v-1} }+2^{b_u^v
}=\sum\limits_{i=0}^{v-u+1} {3^{v-u+1-i}2^{b_u^{u-1+i} }} $.
\end{definition}

Then $b_u^u =a_u $, $b_u^{u+1} =a_u +a_{u+1} $, $B_u^u =3+2^{a_u }$,
$B_u^{u+1} =3^2+3\cdot 2^{a_u }+2^{a_u +a_{u+1} }$, $B_u^v
=3B_u^{v-1} +2^{b_u^v }=\sum\limits_{i=u-1}^v
{3^{v-i}2^{b_u^i }} $. Clearly, $b_1^n $ and
$B_1^{n-1} $are same as $b_n $ and $B_n $, respectively.

\begin{proposition}
$
B_n=3^{n-u+1}B_1^{u-2} +3^{n-1-v}2^{b_{u-1} }B_u^v +2^{b_{v+1}
}B_{v+2}^{n-1} .
$
\end{proposition}
\begin{proof}
By $B_1^{u - 2}  = \sum\limits_{i = 0}^{u - 2}  \,3^{u - 2 - i} 2^{b_i }$
and $B_{v + 2}^{n - 1}  = \sum\limits_{i = v + 1}^{n - 1}
3^{n - 1 - i} 2^{b_{v + 2}^i }$,  we have
\begin{align*}
B_n &= B_1^{n - 1}= \sum\limits_{i = 0}^{n - 1}  \,3^{n - 1 - i} 2^{b_i }  = \sum\limits_{i = 0}^{u - 2}  \,3^{n - 1 - i} 2^{b_i }  + \sum\limits_{i = u - 1}^v  \,3^{n - 1 - i} 2^{b_i }  + \sum\limits_{i = v + 1}^{n - 1}  \,3^{n - 1 - i} 2^{b_i }\\
 &= 3^{n - u + 1} \sum\limits_{i = 0}^{u - 2}  \,3^{u - 2 - i} 2^{b_i }  + 3^{n - 1 - v} 2^{b_{u - 1} } \sum\limits_{i = u - 1}^v  \,3^{v - i} 2^{b_u^i }  + 2^{b_{v + 1} } \sum\limits_{i = v + 1}^{n - 1}  \,3^{n - 1 - i} 2^{b_{v + 2}^i }\\
 &= 3^{n - u + 1} B_1^{u - 2}  + 3^{n - 1 - v} 2^{b_{u - 1} } B_u^v  + 2^{b_{v + 1} } B_{v + 2}^{n - 1}.
\end{align*}
\end{proof}
\begin{definition}
For any $1\leqslant u\leqslant v$, define two integers $x_0^{u,v} $ and $x_{v-u+1}^{u,v} $  such that $2^{b_u^v
}x_{v-u+1}^{u,v} -3^{v-u+1}x_0^{u,v} =B_u^{v-1} $,
$1\leqslant x_0^{u,v}
<2^{b_u^v }$ and $1\leqslant x_{v-u+1}^{u,v} <3^{v-u+1}$.  Further define
$x_1^{u,v} =\cfrac{3x_0^{u,v} +1}{2^{a_u }}$, $x_2^{u,v} =\cfrac{3x_1^{u,v}
+1}{2^{a_{u+1} }}$,{\ldots}, $x_{v-u}^{u,v} =\cfrac{3x_{v-u-1}^{u,v}
+1}{2^{a_{v-1} }}$.
\end{definition}

Clearly, $x_0^{1,n} $and $x_n^{1,n} $ are same as $x_{0}$ and $x_n $ in
Proposition 2.3, respectively.

\begin{proposition}

\begin{enumerate}[label=\upshape(\roman*), leftmargin=*, widest=iii]
\item $x_{v-u+1}^{u,v} =\cfrac{3x_{v-u}^{u,v} +1}{2^{a_v }}; $
\item For any  $0\leqslant k\leqslant  v-u$, $x_{k}^{u,v} =\cfrac{3^{k}x_{0}^{u,v} +B_{u}^{u+k-2}}{2^{b_{u}^{u+k-1} }}$  and\\
     $ x_{v-u+1}^{u,v} =\cfrac{3^{v-u+1-k}x_{k}^{u,v} +B_{u+k}^{v-1}}{2^{b_{u+k}^{v} }}$;
\item   $x_0^{u,v} \leqslant x_0^{u,v+1} $;
\item $\Omega -\lim a_n =x$ if and only if $\lim _{n\to \infty } x_0^{1,n} =x$;
\item  $\Omega -\lim a_n =\infty $  if and only if $\lim _{n\to \infty } x_0^{1,n} =\infty $.
\end{enumerate}
\end{proposition}

\begin{proof} (i) is from Proposition 2.4.  (ii) is from (i) and  Proposition 2.2.

(iii) By Definition 2.7, $2^{b_u^v
}x_{v-u+1}^{u,v} -3^{v-u+1}x_0^{u,v} =B_u^{v-1} $, \\
 $2^{b_u^{v+1}
}x_{v-u+2}^{u,v+1} -3^{v-u+2}x_0^{u,v+1} =B_u^v $. Then\\
$3^{v-u+1}x_0^{u,v} +B_u^{v-1} \equiv 0{\kern 1pt}(\bmod {\kern 1pt}{\kern
1pt}{\kern 1pt}{\kern 1pt}2^{b_u^v })$,
$3^{v-u+2}x_0^{u,v+1} +B_u^v
\equiv 0{\kern 1pt}{\kern 1pt}(\bmod {\kern 1pt}{\kern 1pt}{\kern 1pt}{\kern
1pt}2^{b_u^{v+1} })$. Thus \\
$3^{v-u+1}x_0^{u,v+1}
+B_u^{v-1} \equiv 0{\kern 1pt}{\kern 1pt}(\bmod {\kern 1pt}{\kern
1pt}{\kern 1pt}{\kern 1pt}2^{b_u^v })$ by $B_u^v \mbox{=}3B_u^{v-1}
+2^{b_u^v }$. Hence \\
$x_0^{u,v} \equiv x_0^{u,v+1} {\kern
1pt}{\kern 1pt}(\bmod {\kern 1pt}{\kern 1pt}{\kern 1pt}{\kern 1pt}2^{b_u^v
})$. Therefore \\
$x_0^{u,v} \leqslant x_0^{u,v+1} $ by $1\leqslant x_0^{u,v}
<2^{b_u^v }$ and $1\leqslant x_0^{u,v+1} <2^{b_u^{v+1} }$.

By (iii), $(x_0^{1,n} )_{n\geqslant 1} $  is increasing, then (iv) and (v) hold
trivially.
\end{proof}
Proposition 2.8(iv) shows that if $\Omega  -\lim a_n =x$,  then $x_0^{1,n} =x$  for all sufficiently large $n$.  Proposition 2.8(v) shows the reasonableness of $\Omega -\lim a_n =\infty $.

\section{Periodic E-sequences}
\begin{definition}
\begin{enumerate}[label=\upshape(\roman*), leftmargin=*, widest=iii]
\item $(a_n )_{n\geqslant 1} $ is periodic if there exist two integers \\ $l\geqslant 0$, $r\geqslant 1$ such that $a_n =a_{n+r} $ for all $n>l$;
\item $r$  is
called the period of $(a_n )_{n\geqslant 1} $;
\item $(a_1 \cdots a_l )$ and
$(a_{l+1} \cdots a_{l+r} \cdots )$ are called the non-periodic part and
periodic part of $(a_n )_{n\geqslant 1} $, respectively;
\item  $(a_n )_{n\geqslant 1} $ is called purely periodic if $l=0$ and, eventually periodic if $l>0$;
\item The  E-sequence is denoted by $a_1 \cdots a_l \overline {a_{l+1} \cdots
a_{l+r} } $.
\end{enumerate}
\end{definition}

Throughout the remainder of  this section, define  $s=b_{l+1}^{l+r} $,
$B_r =B_{l+1}^{l+r-1}$ and let $k\geqslant 0$ be an integer.
\begin{proposition}
Let $a_1 \cdots a_l \overline {a_{l+1} \cdots a_{l+r} } $ be a periodic
E-sequence. Then \\
$B_{rk+l} =3^{rk}B_l +2^{b_l }B_r
\cfrac{3^{rk}-2^{sk}}{3^r-2^s}$.
\end{proposition}

\begin{proof} By Proposition 2.6,
$B_{rk+l} =B_1^{rk+l-1}
=3^{rk}B_1^{l-1} +3^{rk-r}2^{b_l }B_{l+1}^{l+r-1} +$\\
$3^{rk-2r}2^{b_{l+r}
}B_{l+r+1}^{l+2r-1} +\cdots +2^{b_{l+rk-r} }B_{l+1+r(k-1)}^{l+rk-1}. $
By
$b_{l+r} =b_l +s,b_{l+2r} =b_l +2s,\cdots,$ \\
$b_{l +rk-r} =b_l +(k-1)s, B_1^{l-1} =B_{l},
B_{l+r+1}^{l+2r-1} =\cdots =B_{l+1+r(k-1)}^{l+rk-1} =B_r $,  we have
\begin{align*}
B_{rk+l} &=3^{rk}B_{l} +3^{rk-r}2^{b_l }B_r +3^{rk-2r}2^{b_l }2^sB_r +\cdots
+2^{b_l }2^{(k-1)s}B_r \\
 &=3^{rk}B_{l} +2^{b_l }B_r (3^{rk-r}2^0+3^{rk-2r}2^s+\cdots
+3^02^{(k-1)s}) \\
&=3^{rk}B_l +2^{b_l }B_r \cfrac{3^{rk}-2^{sk}}{3^r-2^s} .
 \end{align*}
\end{proof}
\begin{proposition}
Let $a_1 \cdots a_l \overline {a_{l+1} \cdots a_{l+r} } $ be a periodic
E-sequence.  By Proposition 2.3,  define  two integers $x_{0}$ and $x_{rk+l} $ such that
$2^{sk+b_l }x_{rk+l} -3^{rk+l}x_{0}=B_{rk+l}$, $1\leqslant x_{0}<2^{sk+b_l }$ and $1\leqslant
x_{rk+l} <3^{rk+l}$. Then there is a constant $K\in \mathbb{N}$, depending
on $a_1 ,\cdots ,a_{l+r} $ such that when $k>K$ and,

\begin{enumerate}[label=\upshape(\roman*), leftmargin=*, widest=iii]
\item if $2^s>3^r$, there is $u_{rk+l} \in \mathbb{Z}$, $0\leqslant u_{rk+l} <(2^s-3^r)3^l$ such that \\
    $x_{0}=  \cfrac{2^{sk+b_l }u_{rk+l} -B_l (2^s-3^r)+2^{b_l }B_r }{(2^s-3^r)3^l}$, $x_{rk+l} =\cfrac{3^{rk}u_{rk+l} +B_r }{2^s-3^r}$;
\item if $3^r>2^s$ there is $u_{rk+l} \in \mathbb{N}$, $1\leqslant u_{rk+l} \leqslant (3^r-2^s)3^l$ such that\\
     $x_{0}= \cfrac{2^{sk+b_l }u_{rk+l} -B_l (3^r-2^s)-2^{b_l }B_r }{(3^r-2^s)3^l}$, $x_{rk+l} =\cfrac{3^{rk}u_{rk+l} -B_r }{3^r-2^s}$.
\end{enumerate}
\end{proposition}

\begin{proof}
\begin{enumerate}[label=\upshape(\roman*), leftmargin=*, widest=iii]
\item  $2^s>3^r$.  By $x_{rk+l} =\cfrac{3^{rk+l}x_{0}+B_{rk+l}
}{2^{sk+b_l }}$, we have \\
$2^{sk+b_l }x_{rk+l} \equiv B_{rk+l}
\,\,(mod\,\,3^{rk+l})$. Then \\
$2^{sk+b_l }x_{rk+l} \equiv 3^{rk}B_l +2^{b_l
}B_r \cfrac{2^{sk}-3^{rk}}{2^s-3^r}\,\,(mod\,\,3^{rk+l})$ by Proposition 3.2. Thus\\
$(2^s-3^r)2^{sk+b_l }x_{rk+l} \equiv (2^s-3^r)3^{rk}B_l
+(2^{sk}-3^{rk})2^{b_l }B_r \,\,(mod\,\,(2^s-3^r)3^{rk+l})$. Hence\\
$
2^{sk+b_l }((2^s-3^r)x_{rk+l} -B_r ) \equiv 3^{rk}((2^s-3^r)B_l -2^{b_l
}B_r )\,\,(mod\,\,(2^s-3^r)3^{rk+l}).
$
Define $u_{rk+l} =\cfrac{(2^s-3^r)x_{rk+l} -B_r }{3^{rk}}$. Then $u_{rk+l} \in
\mathbb{Z}$ and \\
$2^{sk+b_l }u_{rk+l} \equiv (2^s-3^r)B_l -2^{b_l }B_r
\,\,(mod\,\,(2^s-3^r)3^l)$. \\
Hence $x_{rk+l} =\cfrac{3^{rk}u_{rk+l} +B_r
}{2^s-3^r}$ and
\begin{align*}
x_{0}&= \cfrac{2^{sk+b_l }x_{rk+l} -B_{rk+l}}{3^{rk+l}}\\
&= \cfrac{2^{sk+b_l }\cfrac{3^{rk}u_{rk+l} +B_r
}{2^s-3^r}-3^{rk}B_l -2^{b_l }B_r \cfrac{2^{sk}-3^{rk}}{2^s-3^r}}{3^{rk+l}}\\
&=\cfrac{3^{rk}2^{sk+b_l }u_{rk+l} +2^{sk+b_l }B_r -3^{rk}B_l
(2^s-3^r)+3^{rk}2^{b_l }B_r -2^{sk+b_l }B_r }{(2^s-3^r)3^{rk+l}}\\
&=  \cfrac{2^{sk+b_l }u_{rk+l} -B_l (2^s-3^r)+2^{b_l }B_r }{(2^s-3^r)3^l}.
\end{align*}

By $x_{rk+l}
=\cfrac{3^{rk}u_{rk+l} +B_r }{2^s-3^r}<3^{rk+l}$, we have \\
$u_{rk + l}  < \cfrac{{3^{rk + l} (2^s  - 3^r ){\rm{ - }}B_r }}{{3^{rk} }}{\rm{ = }}3^l (2^s  - 3^r ){\rm{ - }}\cfrac{{B_r }}{{3^{rk} }} < 3^l (2^s  - 3^r )$.\\
By $x_{rk+l} =\cfrac{3^{rk}u_{rk+l} +B_r
}{2^s-3^r}>0$, we  have $u_{rk+l} >-\cfrac{B_r }{3^{rk}}$.
Since \\
$\lim _{k\to
\infty } -\cfrac{B_r }{3^{rk}}=0$ and $u_{rk+l} \in \mathbb{Z}$, there is a
constant $K\in \mathbb{N}$, depending on $a_1 ,\cdots ,a_{l+r} $ such that
$u_{rk+l} \geqslant 0$ when $k>K$.

\item  $3^r>2^s$. By $x_{rk+l} =\cfrac{3^{rk+l}x_{0}+B_{rk+l}
}{2^{sk+b_l }}$, we have $$2^{sk+b_l }((3^r-2^s)x_{rk+l} +B_r )  \equiv
3^{rk}((3^r-2^s)B_l +2^{b_l }B_r )\,\,(mod(3^r-2^s)3^{rk+l}).$$  Define $u_{rk+l}
=\cfrac{(3^r-2^s)x_{rk+l} +B_r }{3^{rk}}$. Then
$$u_{rk+l} \in \mathbb{Z},2^{sk+b_l }u_{rk+l} \equiv (3^r-2^s)B_l +2^{b_l }B_r
\,\,(mod(3^r-2^s)3^l).$$ Thus $x_{rk+l} =\cfrac{3^{rk}u_{rk+l} -B_r
}{3^r-2^s}$, $x_{0}=\cfrac{2^{sk+b_l }u_{rk+l} -B_l (3^r-2^s)-2^{b_l
}B_r }{(3^r-2^s)3^l}$. \\
Since $x_{rk+l} =\cfrac{3^{rk}u_{rk+l} -B_r
}{3^r-2^s}>0$, then $u_{rk+l} >\cfrac{B_r }{3^{rk}}$ and thus $1\leqslant u_{rk+l}
$.

By $x_{0}= \cfrac{2^{sk+b_l }u_{rk+l} -B_l (3^r-2^s)-2^{b_l }B_r
}{(3^r-2^s)3^l}<2^{sk+b_l }$, we have \\
$u_{rk+l} <(3^r-2^s)3^l+\cfrac{B_l
(3^r-2^s)+2^{b_l }B_r }{2^{sk+b_l }}$. Since\\
 $\lim _{k\to \infty } \cfrac{B_l
(3^r-2^s)+2^{b_l }B_r }{2^{sk+b_l }}=0$ and $u_{rk+l} \in \mathbb{Z}$, there
is a $K\in \mathbb{N}$ such that $u_{rk+l} \leqslant (3^r-2^s)3^l$ when $k>K$.
\end{enumerate}
\end{proof}

\begin{theorem}
If $3^r>2^s$  then $a_1 \cdots a_l \overline {a_{l+1} a_{l+2}
\cdots a_{l+r-1} a_{l+r} } $ is $\Omega-$divergent.
\end{theorem}

\begin{proof}  By Proposition 3.3(ii),
$x_{0}=\cfrac{2^{sk+b_l }u_{rk+l} -B_l (3^r-2^s)-2^{b_l }B_r }{(3^r-2^s)3^l}$ and \\
$u_{rk+l} \geqslant 1$. Then $x_{0}\to+\infty $ as $k\to \infty. $    Thus  the E-sequence  is $\Omega-$divergent.
\end{proof}

\begin{theorem}
If $a_1 \cdots a_l \overline {a_{l+1} a_{l+2} \cdots a_{l+r-1} a_{l+r} } $
is $\Omega-$convergent to $x$ then $(x_n )_{n\geqslant 0} $ is periodic.
\end{theorem}
\begin{proof}   By  Theorem 3.4,  $2^s>3^r$.  By Proposition 3.3(i),
$$x_{0}=
\cfrac{2^{sk+b_l }u_{rk+l} -B_l (2^s-3^r)+2^{b_l }B_r }{(2^s-3^r)3^l}$$
and $u_{rk+l} \geqslant 0$  for all $k>K$.  Since $x_{0}=x<\infty$  for all sufficiently large $k$ by Proposition 2.8(iv),
then $u_{rk+l} =0$. Thus $x_{0}=\cfrac{2^{b_l }B_r-B_l (2^s-3^r)}{(2^s-3^r)3^l}$  and $x_{rk+l} =\cfrac{B_r}{2^s-3^r}$  for all $k\geq 0$. Hence $(x_n )_{n\geqslant 0} $ is periodic and its
non-periodic part and periodic part are $(x_0 x_1 \cdots x_l )$ and
$\overline {x_{l+1} \cdots x_{l+r} } $, respectively.
\end{proof}

\begin{theorem}
Assume that all non-periodic E-sequence are $\Omega-$divergent. Then
the trajectory of every odd positive integer  is periodic.
\end{theorem}
\begin{proof}
Suppose that  $x$  is an odd positive integer,   $(x_n )_{n\geqslant 0} $  and
$(a_n )_{n\geqslant 1} $  are its trajectory and E-sequence, respectively.  Then
$\Omega -\lim  a_n  =x$.  Thus $(a_n )_{n\geqslant 1} $  is  periodic by the assumption.  Hence $(x_n )_{n\geqslant 0} $ is periodic by  Theorem 3.5.
\end{proof}

\section{Non-periodic E-sequences}
For any real number $\alpha $, $\{\alpha \}$ denotes its fractional part.
The following lemma  is due to  Matthews and Watts (see Lemma 2(b) in [4]).
We present its proof for the reader's convenience.

\begin{lemma}
 Let $(a_n )_{n\geqslant 1} $ be an E-sequence such that
$\Omega -\lim {\kern 1pt}{\kern 1pt}{\kern 1pt}{\kern 1pt}{\kern 1pt}a_n
=x_{0} $ and $(x_n )_{n\geqslant 0} $ is unbounded. Then
$\mathop {\overline {\lim } }\limits_{n\to \infty } \cfrac{b_n }{n}\leqslant
\log _23$.
\end{lemma}

\begin{proof}
From $x_k =\cfrac{3x_{k-1} +1}{2^{a_k }}$,  we have $2^{a_k
}=\cfrac{3x_{k-1} +1}{x_k }$. Then
$$2^{b_n }=\prod\limits_{k=1}^n {2^{a_k }}
=\prod\limits_{k=1}^n {\cfrac{3x_{k-1} +1}{x_k }=\cfrac{x_0 }{x_n
}\prod\limits_{k=1}^n {\cfrac{3x_{k-1} +1}{x_{k-1} }} } =\cfrac{3^nx_0 }{x_n
}\prod\limits_{k=1}^n {(1+\cfrac{1}{3x_{k-1} })}.$$ Thus

$$x_n =\cfrac{3^nx_0 }{2^{b_n }}\prod\limits_{k=1}^n {(1+\cfrac{1}{3x_{k-1} })}
$$ which we call the Matthews and Watts's formula (see Lemma 1(b)  in [4]).

Since $(x_n )_{n\geqslant 1} $ is unbounded, all $x_n $ are distinct. Then
$$1\leqslant x_n \leqslant \cfrac{3^nx_0 }{2^{b_n }}\prod\limits_{k=1}^n
{(1+\cfrac{1}{3k})}.$$
Thus $$0\leqslant \log \cfrac{3^n}{2^{b_n }}+\log x_0
+\sum\limits_{k=1}^n {\log (1+\cfrac{1}{3k})} \leqslant \log 3^n-\log 2^{b_n
}+\log x_0 +\sum\limits_{k=1}^n {\cfrac{1}{3k}}.$$ Hence $$\log 2^{b_n
}\leqslant \log 3^n+\log x_0 +\cfrac{1}{3}\sum\limits_{k=1}^n {\cfrac{1}{k}}
. $$Therefore $$\cfrac{b_n }{n}\leqslant \log _2 3+\cfrac{\log _2 x_0
}{n}+\cfrac{1}{n\log 8}\sum\limits_{k=1}^n {\cfrac{1}{k}}.$$ Then
$$\mathop {\overline {\lim } }\limits_{n\to \infty } \cfrac{b_n }{n}\leqslant
\log _23.$$
\end{proof}
\begin{theorem} Let $(a_n )_{n\geqslant 1} $ be a non-periodic E-sequence
such that $\mathop {\overline {\lim } }\limits_{n\to \infty } \cfrac{b_n
}{n}>\log _23$.  Then\\
$\Omega-\lim a_n=\infty .$
\end{theorem}
\begin{proof}Suppose that $\Omega -\lim a_n =x_0 $ for some positive integer $x_0 $. It
follows from Lemma 4.1 and $\mathop {\overline {\lim } }\limits_{n\to \infty }
\cfrac{b_n }{n}>\log _23$ that $(x_n )_{n\geqslant 0} $ is bounded.
Then $(x_n )_{n\geqslant 0} $ is periodic. Thus $(a_n )_{n\geqslant
1} $ is periodic, which contradicts the non-periodicity of $(a_n
)_{n\geqslant 1} $. Hence $\Omega -\lim a_n =\infty $.
\end{proof}

The following lemma  is the well-known Wendel's inequality (see~\cite{WJ}). Lemma 4.4
is a  consequence of an easy calculation.
\begin{lemma}
Let $x$ be a positive real number and let $s\in (0,1)$.
Then $\cfrac{\Gamma (x+s)}{\Gamma
(x)}\leqslant x^s$.
\end{lemma}
\begin{lemma} Let $a$ and $b$ be two integers with $a\geqslant  1$ and $a\nmid b$. Then
$\prod\limits_{k=0}^n {(1+\frac{z}{ak+b})} =\cfrac{\Gamma (\frac{b}{a})\Gamma
(\frac{b+z}{a}+n+1)}{\Gamma (\frac{b+z}{a})\Gamma (\frac{b}{a}+n+1)}$.
\end{lemma}

\begin{lemma}$\prod\limits_{
1\leqslant k<3n, \,\,
 k\equiv 1,5(\bmod {\kern 1pt}{\kern 1pt}6)
 } {(1+\frac{1}{3k})} <1.5n^{\frac{1}{9}}$ for all $n\geqslant 1$.
\end{lemma}
\begin{proof}
Let $2\vert n$. Then $$\prod\limits_{k=0}^{\frac{n}{2}-1}
{(1+\frac{1}{3(6k+1)})} =\frac{\Gamma (\frac{1}{6})\Gamma
(\frac{n}{2}+\frac{2}{9})}{\Gamma (\frac{2}{9})\Gamma
(\frac{n}{2}+\frac{1}{6})}\leqslant \frac{\Gamma (\frac{1}{6})}{\Gamma
(\frac{2}{9})}(\frac{n}{2}+\frac{1}{6})^{\frac{1}{18}}$$ and

$$\prod\limits_{k=0}^{\frac{n}{2}-1} {(1+\frac{1}{3(6k+5)})} =\frac{\Gamma
(\frac{5}{6})\Gamma (\frac{n}{2}+\frac{8}{9})}{\Gamma (\frac{8}{9})\Gamma
(\frac{n}{2}+\frac{5}{6})}\leqslant \frac{\Gamma (\frac{5}{6})}{\Gamma
(\frac{8}{9})}(\frac{n}{2}+\frac{5}{6})^{\frac{1}{18}}$$ by the Wendel's
inequality. Thus

$\prod\limits_{
 1\leqslant k<3n,\,\, k\equiv 1,5(\bmod {\kern 1pt}{\kern 1pt}6) } {(1+\frac{1}{3k})} =\prod\limits_{k=0}^{\frac{n}{2}-1}
{(1+\frac{1}{3(6k+1)})\prod\limits_{k=0}^{\frac{n}{2}-1}
{(1+\frac{1}{3(6k+5)})} } \leqslant $
\[
\frac{\Gamma (\frac{1}{6})\Gamma (\frac{5}{6})}{\Gamma (\frac{2}{9})\Gamma
(\frac{8}{9})}(\frac{n}{2}+\frac{5}{6})^{\frac{1}{18}}(\frac{n}{2}+\frac{1}{6})^{\frac{1}{18}}\leqslant
\mbox{1.4196}(\frac{n^2}{3})^{\frac{1}{18}}<\mbox{1.5}n^{\frac{1}{9}}.
\]
Let $2\nmid n$. Then $$\prod\limits_{k=0}^{\frac{n+1}{2}-1}
{(1+\frac{1}{3(6k+1)})} =\frac{\Gamma (\frac{1}{6})\Gamma
(\frac{n}{2}+\frac{13}{18})}{\Gamma (\frac{2}{9})\Gamma
(\frac{n}{2}+\frac{2}{3})}\leqslant \frac{\Gamma (\frac{1}{6})}{\Gamma
(\frac{2}{9})}(\frac{n}{2}+\frac{2}{3})^{\frac{1}{18}}$$ and

$$\prod\limits_{k=0}^{\frac{n+1}{2}-2} {(1+\frac{1}{3(6k+5)})} =\frac{\Gamma
(\frac{5}{6})\Gamma (\frac{n}{2}+\frac{7}{18})}{\Gamma (\frac{8}{9})\Gamma
(\frac{n}{2}+\frac{1}{3})}\leqslant \frac{\Gamma (\frac{5}{6})}{\Gamma
(\frac{8}{9})}(\frac{n}{2}+\frac{1}{3})^{\frac{1}{18}}$$ by the Wendel's
inequality.  Thus

$\prod\limits_{1\leqslant k<3n,\,\,\\
 k\equiv 1,5(\bmod {\kern 1pt}{\kern 1pt}6)} {(1+\frac{1}{3k})} =\prod\limits_{k=0}^{\frac{n+1}{2}-1}
{(1+\frac{1}{3(6k+1)})\prod\limits_{k=0}^{\frac{n+1}{2}-2}
{(1+\frac{1}{3(6k+5)})} } \leqslant $
\[
\frac{\Gamma (\frac{1}{6})\Gamma (\frac{5}{6})}{\Gamma (\frac{2}{9})\Gamma
(\frac{8}{9})}(\frac{n}{2}+\frac{2}{3})^{\frac{1}{18}}(\frac{n}{2}+\frac{1}{3})^{\frac{1}{18}}<1.5n^{\frac{1}{9}}.
\]
\end{proof}
\begin{theorem}
Let $2^{b_n }x_n -3^nx_0 =B_n $ such that $1\leqslant x_0
<2^{b_n },  1\leqslant x_n <3^n$, $3\nmid x_0 $,  and $x_0 ,\cdots ,x_{n-1} $ are distinct
integers. Then $x_0 >\cfrac{B_n }{3^n(1.5n^{\frac{1}{9}}-1)}$.
\end{theorem}
\begin{proof}
From the Matthews and Watts's formula and Lemma 4.5, we
have
$$\cfrac{2^{b_n }x_n }{3^nx_0 }=\prod\limits_{k=1}^n {(1+\cfrac{1}{3x_{k-1}
})\leqslant } \prod\limits_{
 1\leqslant k<3n,\,\,
 k\equiv 1,5(\bmod {\kern 1pt}{\kern 1pt}6)
} {(1+\cfrac{1}{3k})} <1.5n^{\frac{1}{9}}.$$
Then $\cfrac{3^nx_0
+B_n }{3^nx_0 }<1.5n^{\frac{1}{9}}$. Thus $x_0 >\cfrac{B_n
}{3^n(1.5n^{\frac{1}{9}}-1)}$.
\end{proof}

\begin{corollary}
 Let $\theta \geqslant \mbox{log}_23 $ be
an irrational number.  Define
$a_n =[n\theta ]-[(n-1)\theta ]$.  Then $\Omega
-\lim {\kern 1pt}{\kern 1pt}{\kern 1pt}a_n =\infty $.
\end{corollary}

\begin{proof}
Let $\theta =\mbox{log}_23 $. Then
$\cfrac{B_n }{3^n}=\sum\limits_{k=1}^{n}
{\frac{2^{[(k-1)\log_23]}}{3^k}} > \cfrac{n}{8}$ by
$\cfrac{2^{[(k-1)\log_23]}}{3^k}>\cfrac{1}{8}$. Thus\\
$\cfrac{B_n }{3^n(1.5n^{\frac{1}{9}}-1)}>\cfrac{n}{8(1.5n^{\frac{1}{9}}-1)}\to
\infty $, as $n\to \infty $. Hence $\Omega -\lim {\kern 1pt}{\kern
1pt}{\kern 1pt}a_n =\infty $ by Theorem 4.6.

Let $\theta >\log_23 $. Then $\mathop {\lim
}\limits_{n\to \infty } {\kern 1pt}{\kern 1pt}{\kern 1pt}{\kern
1pt}\cfrac{b_n }{n}=\mathop {\lim }\limits_{n\to \infty } {\kern 1pt}{\kern
1pt}{\kern 1pt}{\kern 1pt}\cfrac{\mbox{[n}\theta ]}{n}=\theta
>\log_23 $. Since $\theta $ is an irrational number,
$(a_n )_{n\geqslant 1} $ is non-periodic. Thus $\Omega -\lim {\kern
1pt}{\kern 1pt}{\kern 1pt}a_n =\infty $ by Theorem 4.2.
\end{proof}

\begin{lemma} Let $x$  and $n$ be two positive integers.  Then  (i) $\prod\limits_{k=0}^{n-1}
{(1+\cfrac{1}{3(x+k)})} \le 1+\cfrac{n}{3x}; $ (ii) $
\prod\limits_{k=0}^{n-1} {(1+\cfrac{1}{3(x -k)})} \ge 1+\cfrac{n}{3x }
$  for $x \ge n$;
(iii) $
\prod\limits_{k=0}^{n-1} {(1+\cfrac{1}{3(x -k)})} >\cfrac{3x }{3x -n}
$   for  $x\ge n\ge 2$.
\end{lemma}
\begin{proof}
(i) The proof is by induction on $n$. For the base step, let
$n=1$ then $\prod\limits_{k=0}^{n-1} {(1+\cfrac{1}{3(x+k)})}
=1+\cfrac{1}{3x}=1+\cfrac{n}{3x}$.
For the induction step, assume that $\prod\limits_{k=0}^{n-1}
{(1+\cfrac{1}{3(x+k)})} \le 1+\cfrac{n}{3x}$.  Then
$\prod\limits_{k=0}^n {(1+\cfrac{1}{3(x+k)})} \le
(1+\cfrac{n}{3x})(1+\cfrac{1}{3(x+n)})=1+\cfrac{n}{3x}+\cfrac{1}{3(x+n)}+\cfrac{n}{9x(x+n)}\le 1+\cfrac{n+1}{3x}.$ Thus
the inequality holds for all $n\ge 1$.
The proof of (ii) is similar to that of (i) and omitted.

(iii) Let $n=2$.  Since $3x \cdot 3x-2\cdot 3x -3x +2>3x \cdot 3x -3\cdot 3x$ then
$\cfrac{3x -1}{3x \cdot 3(x -1)}>\cfrac{1}{3x -2}$. Thus $1+\cfrac{1}{3x }+\cfrac{1}{3(x
-1)}+\cfrac{1}{3x \cdot 3(x -1)}>\cfrac{3x -2+2}{3x
-2}=1+\cfrac{2}{3x -2}$. Hence  \\
$(1+\cfrac{1}{3x })(1+\cfrac{1}{3(x
-1)})>\cfrac{3x }{3x -2}$. Therefore  $\prod\limits_{k=0}^{n-1} {(1+\cfrac{1}{3(x
-k)})} >\cfrac{3x }{3x -n}$.

Assume that  $\prod\limits_{k=0}^{n-1} {(1+\cfrac{1}{3(x
-k)})} >\cfrac{3x }{3x -n}$.  Since  $ (3x -3n+1)(3x -n-1)>(3x -n)(3x -3n)$  then  $ \cfrac{3x (3x -3n)+3x }{(3x -n)(3x
-3n)}>\cfrac{3x }{3x -n-1}$. Thus $ \prod\limits_{k=0}^n {(1+\cfrac{1}{3(x -k)})} >\cfrac{3x
}{3x -n}(1+\cfrac{1}{3(x -n)})= \cfrac{3x }{3x -n}+\cfrac{3x }{(3x -n)(3x
-3n)}>\cfrac{3x }{3x -n-1}$.
\end{proof}

\begin{lemma} Let $2^{b_n }x_n -3^nx_0 =B_n $ such that $1\le x_0 <2^{b_n
}$, $1\le x_n <3^n$, $x_i \ne x_j $ for all $0\le i<j\le n-1$.     Then
(i) $\cfrac{B_n }{3^n}\le \cfrac{n}{3}$  if $x_k >x_0
$ for all  $1\le k\le n-1$;
(ii)  $\cfrac{B_n }{2^{b_n }}<\cfrac{n}{3}$ if $x_n <x_k
$ for all $0\le k\le n-1$;
(iii) $\cfrac{B_n }{2^{b_n }}>\cfrac{n}{3}$ if $x_n >x_i
$ for all $0\le i\le n-1$;
(iv)  $\cfrac{B_n }{3^n}\ge \cfrac{n}{3}$ if  $x_0 >x_k
$ for all $1\le k\le n$.
\end{lemma}

\begin{proof}
(i) From $\cfrac{2^{b_n }x_n }{3^nx_0 }=\prod\limits_{k=0}^{n-1}
{(1+\cfrac{1}{3x_k })} $, we have
\[
1+\cfrac{B_n }{3^nx_0 }=\prod\limits_{k=0}^{n-1} {(1+\cfrac{1}{3x_k })} \le
\prod\limits_{k=0}^{n-1} {(1+\cfrac{1}{3(x_0 +k)})} .
\]
Then $1+\cfrac{B_n }{3^nx_0 }\le 1+\cfrac{n}{3x_0 }$ by Lemma 4.8(i). Thus $\cfrac{B_n
}{3^n}\le \cfrac{n}{3}$.

(ii)  From $\cfrac{2^{b_n }x_n }{3^nx_0 }=\prod\limits_{k=0}^{n-1}
{(1+\cfrac{1}{3x_k })} $,  we have
\[
\cfrac{2^{b_n }x_n -B_n }{2^{b_n }x_n }=\prod\limits_{k=0}^{n-1}
{(1+\cfrac{1}{3x_k })^{-1}} \ge \prod\limits_{k=0}^{n-1} {(1+\cfrac{1}{3(x_n
+k)})^{-1}}.
\]
Then $1-\cfrac{B_n }{2^{b_n }x_n }\ge \prod\limits_{k=0}^{n-1}
{(1+\cfrac{1}{3(x_n +k)})^{-1}}\ge (1+\cfrac{n}{3x_n })^{-1}$ by Lemma 4.8(i). Thus $$\cfrac{B_n
}{2^{b_n }x_n }\le 1-(1+\cfrac{n}{3x_n })^{-1}=\cfrac{n}{3x_n +n}.$$
Hence $\cfrac{B_n }{2^{b_n }}\le \cfrac{nx_n }{3x_n +n}< \cfrac{n}{3}$.

(iii) Let $n=1$. Then $x_1 =\cfrac{3x+1}{2^{a_1 }}>x$.
Thus $(3-2^{a_1 })x+1>0$. Hence $a_1 =1$. Therefore $\cfrac{B_n }{2^{b_n
}}=\cfrac{B_1 }{2^{b_1 }}=\cfrac{1}{2}>\cfrac{1}{3}=\cfrac{n}{3}$.

Let $x_n \ge n\ge 2$. By  Lemma 4.8(iii),   we have  $\cfrac{2^{b_n }x_n }{3^nx_0
}=\prod\limits_{k=0}^{n-1} {(1+\cfrac{1}{3x_k })} \ge
\prod\limits_{k=0}^{n-1} {(1+\cfrac{1}{3(x_n -k)})} >\cfrac{3x_n }{3x_n -n}$. Then
 $\cfrac{2^{b_n }x_n }{2^{b_n }x_n -B_n }>\cfrac{3x_n }{3x_n -n}$. Thus
$\cfrac{2^{b_n }x_n -B_n }{2^{b_n }x_n }<\cfrac{3x_n -n}{3x_n }$.  Hence
$\cfrac{B_n }{2^{b_n }}>\cfrac{n}{3}.$

(iv) By  Lemma 4.8(ii),  we have
\[1+\cfrac{B_n }{3^nx_0 }=\cfrac{2^{b_n }x_n }{3^nx_0 }=\prod\limits_{k=0}^{n-1}
{(1+\cfrac{1}{3x_k })} \ge
\prod\limits_{k=0}^{n-1} {(1+\cfrac{1}{3(x_0 -k)})} \ge 1+\cfrac{n }{3 x_0 }.
\]
Then  $\cfrac{B_n
}{3^n}\ge \cfrac{n}{3}$.
\end{proof}

A direct consequence of Lemma 4.9  is the following theorem,  which may imply something unknown.

\begin{theorem} Let $2^{b_n }x_n -3^nx_0 =B_n $ such that $1\le x_0 <2^{b_n
}$, $1\le x_n <3^n$, $x_i \ne x_j $ for all $0\le i<j\le n-1$.
Then

(i) $\cfrac{B_n }{3^n}>\cfrac{n}{3}$ implies $x_k \le x_0 $ for some $1\le
k\le n-1$;

(ii) $\cfrac{B_n }{3^n}<\cfrac{n}{3}$ implies $x_0 \le x_k $ for some $1\le
k\le n$;

(iii) $\cfrac{B_n }{2^{b_n }}\le \cfrac{n}{3}$ implies $x_n \le x_i $ for some $0\le
i\le n-1$;

(iv) $\cfrac{B_n }{2^{b_n }}\ge \cfrac{n}{3}$ implies $x_n \ge x_k $ for some
$0\le k\le n-1$.
\end{theorem}

\begin{theorem}
Let $(a_n )_{n\geqslant 1} $ be an E-sequence such that (i) $3^n>2^{b_n }$
for all $n\in \mathbb{N}$; (ii) There is a constant $c>\log_23 $ such that
there are infinitely many distinct pairs $(k,l)$ of positive integers such that $l>kc$, $a_{k+1} =\cdots =a_l =1$. Then
$\Omega -\lim a_n =\infty
$.
\end{theorem}

\begin{proof}It follows from (i) that $B_{n} < 3^nn$ for all
$n\in \mathbb{N}$ by induction on $n$.  $B_{k+1}^{l-1}
=3^{l-k}-2^{l-k}$ by Proposition 3.2.

Let $x_l^{1,l} =\cfrac{3^lx_0^{1,l} +B_1^{l-1} }{2^{b_l }}$, $1\leqslant
x_0^{1,l} <2^{b_l }$,$1\leqslant x_l^{1,l} <3^l$. Then $x_k^{1,l}
=\cfrac{3^kx_0^{1,l} +B_1^{k-1} }{2^{b_k }}$, $x_l^{1,l}
=\cfrac{3^{l-k}x_k^{1,l} +B_{k+1}^{l-1} }{2^{b_{k+1}^l }}$ by Proposition 2.8(ii). By $B_{k+1}^{l-1}
=3^{l-k}-2^{l-k}$, $2^{b_{k+1}^l }=2^{l-k}$, \\
we have
$2^{l-k}(x_l^{1,l}
+1)=3^{l-k}(x_k^{1,l} +1)$. Thus $x_k^{1,l} =2^{l-k}w-1$ for some
$1\leqslant w$.
Hence $x_k^{1,l} =\cfrac{3^kx_0^{1,l} +B_1^{k-1} }{2^{b_k
}}=2^{l-k}w-1$.  Therefore \\
$x_0^{1,l} =\cfrac{2^{l-k}2^{b_k }w-2^{b_k
}-B_1^{k-1} }{3^k}\geqslant \cfrac{2^l}{3^k}2^{b_k -k}-1-k\geqslant
(\cfrac{2^c}{3})^k2^{b_k -k}-1-k$.
If there are only finitely many distinct $k$ in all pairs $(k,l)$,
$x_0^{1,l} \geqslant \cfrac{2^l}{3^k}2^{b_k -k}-1-k\to \infty
$, as $ l\to \infty $; otherwise  $x_0^{1,l} \geqslant (\cfrac{2^c}{3})^k2^{b_k -k}-1-k\to \infty$, as $ k\to \infty $.
Then $\Omega
-\lim a_n =\infty $.
\end{proof}

\begin{corollary}
Let $(a_n )_{n\geqslant 1} $be the E-sequence $12121112\cdots $, where $a_n
=2$ if ${\kern 1pt}{\kern 1pt}n\in \{2^1,2^2,2^3,\cdots \}$ and $a_n =1$
otherwise. Then $\Omega -\lim a_n =\infty $.
\end{corollary}

\begin{proof}Take $c=\cfrac{7}{4}>\log_23 $, $k=2^m$ and $l=2^{m+1}-1$.
Then $a_{k+1} =\cdots =a_l =1$, $l>kc$ for all
${\kern 1pt}m\geqslant 3.$ Thus $\Omega -\lim a_n =\infty $ by Theorem 4.11.
\end{proof}
\begin{theorem}
Let $(a_n )_{n\geqslant 1} $ be an E-sequence such that (i) $3^n> 2^{b_n }$
for all $n\in \mathbb{N}$; (ii) there is a constant $c>\log_23 $ such that
there are infinitely many distinct pairs $(r,l)$ of positive integers such that $l>r$, $b_{l+r} >lc$, $a_{l+k} =a_k $ for all $1\leqslant k\leqslant r$, i.e., $(a_1 \cdots a_r )a_{r+1}\cdots a_l (a_{l+1} \cdots a_{l+r} )$ is contained in $ (a_n )_{n\geqslant 1} $. Then
$\Omega -\lim a_n =\infty
$.
\end{theorem}

\begin{proof} Let $x_{l+r}^{1,l+r} =\cfrac{3^{l+r}x_0^{1,l+r} +B_1^{l+r-1}
}{2^{b_1^{l+r} }}$, $1\leqslant x_0^{1,l+r} <2^{b_1^{l+r} }$, $1\leqslant
x_{l+r}^{1,l+r} <3^{l+r}$. Then $x_l^{1,l+r} =\cfrac{3^lx_0^{1,l+r}
+B_1^{l-1} }{2^{b_1^l }}$, $x_{l+r}^{1,l+r} =\cfrac{3^rx_l^{1,l+r}
+B_{l+1}^{l+r-1} }{2^{b_{l+1}^{l+r} }}=\cfrac{3^rx_l^{1,l+r} +B_1^{r-1}
}{2^{b_1^r }}$ by Proposition 2.8(ii). By $3^l>2^{b_1^l }$, we have $x_l^{1,l+r} >x_0^{1,l+r}$.

Let $x_r^{1,r} =\cfrac{3^rx_0^{1,r} +B_1^{r-1} }{2^{b_1^r }}$, $1\leqslant
x_0^{1,r} <2^{b_1^r }$, $1\leqslant x_r^{1,r} <3^r$. Then \\
$x_0^{1,r} \equiv
x_l^{1,l+r} (\bmod {\kern 1pt}{\kern 1pt}{\kern 1pt}{\kern 1pt}2^{b_1^r })$.
By Proposition 2.8(iii),  we have $x_0^{1,l+r} \geqslant x_0^{1,r} $. \\
Let $x_l^{1,l+r} =2^{b_1^r }u+x_0^{1,r} $. Then $u\geqslant 1$ by
$x_l^{1,l+r} >x_0^{1,l+r} \geqslant x_0^{1,r} $. Thus
$$x_0^{1,l+r} =\cfrac{2^{b_1^l }2^{b_1^r }u+2^{b_1^l }x_0^{1,r} -B_1^{l-1}
}{3^l}\geqslant \cfrac{2^{b_1^{l+r} }}{3^l}-l\geqslant
(\cfrac{2^c}{3})^l-l\to \infty, \,\,as\,\, l\to \infty.$$
Hence  $\Omega-\lim a_n =\infty $.
\end{proof}

\begin{theorem}
Let $1\leqslant \theta <\log_23 $  and define $a_n =[n\theta ]-[(n-1)\theta
]$. Then \\
$\Omega -\lim a_n =\infty $.
\end{theorem}
\begin{proof}
 If $\theta $ is a rational number then $(a_n )_{n\geqslant
1} $ is purely periodic and the result follows from Theorem 3.4. Let $\theta
$ be an irrational number in the following. By Hurwitz theorem there are
infinite convergents $\cfrac{s}{r}$ of $\theta $ such that $\vert \theta
-\cfrac{s}{r}\vert <\cfrac{1}{\sqrt 5 r^2}$. There are two cases to be
considered.

\textbf{Case 1} There are infinite convergents $\cfrac{s}{r}$ of $\theta $
such that $0<\theta -\cfrac{s}{r}<\cfrac{1}{\sqrt 5 r^2}$. We prove that
$[\theta n]=[\cfrac{s}{r}n]$ for all  $1\leqslant n\leqslant [\sqrt 5 r]$. By
$1\leqslant n\leqslant \left[ {\sqrt 5 r} \right]$, we have $0<\theta
n-\cfrac{s}{r}n<\cfrac{n}{\sqrt 5 r^2}<\cfrac{\sqrt 5 r}{\sqrt 5
r^2}=\cfrac{1}{r}.
$ Then
$0\leqslant \{\cfrac{s}{r}n\}<\theta
n-[\cfrac{s}{r}n]<\cfrac{1}{r}+\{\cfrac{s}{r}n\}\leqslant 1.$
Thus $0<\theta n-[\cfrac{s}{r}n]<1.$
Hence $[\theta n]=[\cfrac{s}{r}n]$. Then we have the
following periodic table for $(a_n )_{1\leqslant n\leqslant \left[ {\sqrt 5
r} \right]} $.

\begin{table}[htbp]
\begin{center}
\begin{tabular}{|p{34pt}|p{35pt}|p{24pt}|p{50pt}|p{24pt}|p{26pt}|}
\hline
$a_1 $&
$a_2 $&
$\cdots $&
$a_{\left[ {\sqrt 5 r-2r} \right]} $&
$\cdots $&
$a_r $ \\
\hline
$a_{r+1} $&
$a_{2+r} $&
$\cdots $&
$a_{\left[ {\sqrt 5 r-r} \right]} $ &
$\cdots $&
$a_{2r} $  \\
\hline
$a_{2r+1} $&
$a_{2+2r} $&
$\cdots $&
$a_{\left[ {\sqrt 5 r} \right]} $&
&
 \\
\hline
\end{tabular}
\label{tab1}
\end{center}
\end{table}

By Proposition 3.3(ii),
$x_0^{1,2r} =\cfrac{2^{2[r\theta ]}u_{2r} -B_r
}{3^r-2^{[r\theta ]}}
$
for some $u_{2r} \geqslant 1$. \\
By
 $B_r =\sum\limits_{i=0}^{r-1} {3^{r-1-i}2^{b_i }}
=3^{r-1}\sum\limits_{i=0}^{r-1} {\cfrac{2^{b_i }}{3^i}} \leqslant
3^{r-1}\sum\limits_{i=0}^{r-1} {\cfrac{2^{[i\theta ]}}{3^i}} \leqslant
3^{r-1}\sum\limits_{i=0}^{r-1} {\cfrac{2^{i\theta }}{3^i}}
=\cfrac{3^r}{3}\cfrac{1-(\cfrac{2^\theta }{3})^r}{1-\cfrac{2^\theta
}{3}}=\cfrac{3^r-2^{r\theta }}{3-2^\theta }\leqslant \cfrac{3^r}{3-2^\theta
}$,  we have
\[x_0^{1,2r} \geqslant \cfrac{2^{2[r\theta ]}-B_r }{3^r-2^{[r\theta
]}}\geqslant \cfrac{4^{r\theta -1}-\cfrac{3^r}{3-2^\theta }}{3^r-2^{r\theta -1}}=
\cfrac{{\cfrac{1}{4}(\cfrac{{4^\theta  }}{3})^r  - {\textstyle{1 \over {3 - 2^\theta  }}}}}{{1 - \cfrac{1}{2}(\cfrac{{2^\theta  }}{3})^r }} .\]
Thus
$x_0^{1,2r} \to \infty $, as $r\to \infty $. Hence $\Omega
-\lim a_n =\infty $ .

\textbf{Case 2 }There are infinite convergents $\cfrac{s}{r}$ of $\theta $
such that $0<\cfrac{s}{r}-\theta <\cfrac{1}{\sqrt 5 r^2}$.

Firstly, we prove $[\theta n]=[\cfrac{s}{r}n]$ for all $1\leqslant n\leqslant
[\sqrt 5 r]$, $n\notin \{r,2r\}$. By $0<\cfrac{s}{r}-\theta <\cfrac{1}{\sqrt 5
r^2}$, we have $\cfrac{s}{r}-\cfrac{1}{\sqrt 5 r^2}<\theta <\cfrac{s}{r}$. Then
$\cfrac{s}{r}n-[\cfrac{s}{r}n]-\cfrac{n}{\sqrt 5 r^2}<\theta
n-[\cfrac{s}{r}n]<\cfrac{s}{r}n-[\cfrac{s}{r}n]<1$. By $1\leqslant n\leqslant
[\sqrt 5 r]$, $n\notin \{r,2r\}$, we have $ 0<\cfrac{1}{r}-\cfrac{n}{\sqrt 5
r^2}\leqslant \cfrac{s}{r}n-[\cfrac{s}{r}n]-\cfrac{n}{\sqrt 5 r^2}$. Then $
0<\theta n-[\cfrac{s}{r}n]<1$. Thus $[\theta n]=[\cfrac{s}{r}n]$.

Secondly, we prove $[r\theta ]=s-1$, $[2r\theta ]=2s-1$. By $1\leqslant n$,
$0<\cfrac{s}{r}-\theta <\cfrac{1}{\sqrt 5 r^2}$, we have $-\cfrac{n}{\sqrt 5
r^2}+\cfrac{s}{r}n<n\theta <\cfrac{s}{r}n$. By $n<\sqrt 5 r$, we have $
-1<-\cfrac{1}{r}<-\cfrac{n}{\sqrt 5 r^2}$. Then
$-1+\cfrac{s}{r}n<-\cfrac{n}{\sqrt 5 r^2}+\cfrac{s}{r}n<n\theta <\cfrac{s}{r}n$.
By taking $n=r,2r$, we have $[r\theta ]=s-1$, $[2r\theta ]=2s-1$.

Let $2\leqslant j\leqslant r-1$ then $r+2\leqslant r+j\leqslant 2r-1$ and
$r+1\leqslant r+j-1\leqslant 2r-2$. Thus $a_{r+j} =[\theta (r+j)]-[\theta
(r+j-1)]=[\cfrac{s}{r}(r+j)]-[\cfrac{s}{r}(r+j-1)]=[s+\cfrac{s}{r}j]-[s+\cfrac{s}{r}(j-1)]=[\cfrac{s}{r}j]-[\cfrac{s}{r}(j-1)]=a_j
$.

Let $2\leqslant j\leqslant [\sqrt 5 r]-2r$. Then $2r+2\leqslant
2r+j\leqslant [\sqrt 5 r]$ and $2r+1\leqslant 2r+j-1\leqslant [\sqrt 5
r]-1$. Thus $a_{2r+j} =[\theta (2r+j)]-[\theta
(2r+j-1)]=[\cfrac{s}{r}(2r+j)]-[\cfrac{s}{r}(2r+j-1)]=[\cfrac{s}{r}j]-[\cfrac{s}{r}(j-1)]=a_j
$.

By easy calculation, we have $a_r=a_{2r}=1$, $a_{r+1} =a_{2r+1}
=2$.

Then we have the following periodic table for $(a_n )_{1\leqslant n\leqslant
\left[ {\sqrt 5 r} \right]} $.
\begin{table}[htbp]
\begin{center}
\begin{tabular}
{|c|c|c|c|c|c|c|c|}
\hline
$a_1 $&$a_2 $&$a_3 $&$\cdots $&$a_{\left[ {\sqrt 5 r} \right]-2r} $&$\cdots $&$a_r $&$a_{r+1} $ \\
\hline
&$a_{2+r} $&$a_{3+r} $&$\cdots $&$a_{\left[ {\sqrt 5 r} \right]-r} $ &
$\cdots $&$a_{2r} $&$a_{2r+1} $  \\
\hline
&$a_{2+2r} $&$a_{3+2r} $&$\cdots $&$a_{\left[ {\sqrt 5 r} \right]} $&
&&
 \\
\hline
\end{tabular}
\label{tab2}
\end{center}
\end{table}
Since $\theta <\log_23 $, we then take all convergents $\cfrac{s}{r}$ of
$\theta $ such that $\cfrac{s}{r}<\log_23 $ and thus $2^s<3^r$. By $a_1
=1$, $b_2^{r+1} =[r\theta ]+1=s$ and Proposition 3.3(ii), we have
\[
x_0^{1,2r+1} = \cfrac{2^{2s+1}u_{2r+1} -(3^r-2^s)-2B_2^r }{3(3^r-2^s)}
\]
for some $u_{2r+1} \geqslant 1$ . By $B_2^r =\sum\limits_{i=0}^{r-1}
{3^{r-1-i}2^{b_2^{i+1} }} =3^{r-1}\sum\limits_{i=0}^{r-1}
{\cfrac{2^{[i\theta +\theta ]-1}}{3^i}} \leqslant 3^{r-1}2^{\theta
-1}\sum\limits_{i=0}^{r-1} {\cfrac{2^{i\theta }}{3^i}} =2^{\theta
-1}\cfrac{3^r}{3}\cfrac{1-(\cfrac{2^\theta }{3})^r}{1-\cfrac{2^\theta
}{3}}=2^{\theta -1}\cfrac{3^r-2^{r\theta }}{3-2^\theta }\leqslant C3^r$,
where $C=\cfrac{2^{\theta -1}}{3-2^\theta }$,  we have
$$x_0^{1,2r+1} \geqslant
\cfrac{2}{3}\cfrac{4^{[r\theta ]+1}-C3^r}{3^r-2^s}-\cfrac{1}{3}\geqslant
\cfrac{2}{3}\cfrac{4^{r\theta }-C3^r}{3^r-2^s}-\cfrac{1}{3}\geqslant
\cfrac{2}{3}\cfrac{4^{r\theta
}-C3^r}{3^r}-\cfrac{1}{3}=\cfrac{2}{3}(\cfrac{4^\theta
}{3})^r-\cfrac{2}{3}C-\cfrac{1}{3}. $$Thus $\lim _{r\to \infty } x_0^{1,2r+1}
=\infty $. Hence $ \Omega -\lim a_n =\infty $.
\end{proof}

\section{Concluding Remarks and open problems}
The results on non-periodic E-sequences in Section 4 are based on the theory
of periodic E-sequences in Section 3 and the Matthews and Watts's formula. Currently, we have no other way to
tackle with non-periodic E-sequences.  We can obtain various generalizations
and analogues of Theorem 4.2,  4.6,  4.10, 4.11   and 4.13. But we need good problems to make some progress.

One seemingly simple problem which we are not able to prove is whether
 $(a_n )_{n\geqslant 1} $  is divergent, where $a_n =2$ if
$n\in \{2^2,3^2,4^2,\ldots \}$ and $a_n =1$ otherwise, i.e., $(a_n )_{n\geqslant 1} $  is $111211112\ldots $.

Another interesting problem is whether $(a_n)_{n\geqslant 1}$  with infinitely many $n$ satisfying $b_n>n\log_23$  is $\Omega-$divergent.   By virtue of  Theorem 4.2, we only need to consider the  case of $\mathop {\overline {\lim } }\limits_{n\to \infty } \cfrac{b_n }{n}=\log_23$.  Theorem 4.6  answers  the problem  if $\cfrac{B_n }{3^n(1.5n^{\frac{1}{9}}-1)}\to
\infty $, as $n\to \infty $.  Currently, we don't know how to tackle with the other cases of the problem.

Conjecture 1.2(ii) is also important  in some sense.

\section*{Availability of data and materials}
\noindent Not applicable.
\section*{Competing interests}
\noindent The author declares to have no competing interests.
\section*{Funding}
\noindent This work is supported by the National Foundation of Natural Sciences of China (Grant No: 61379018
{\&}61662044{\&}11571013).
\section*{Authors¡¯ contributions}
\noindent The author completed the work alone, and read and approved the final manuscript.
\section*{Acknowledgements}
\noindent I am greatly indebted to Prof. JunDe Miao for his constant encouragement
while I was working on the problem.

\section*{References}
\bibliographystyle{elsarticle-harv}

\end{document}